\documentclass[12pt]{article}
\usepackage{amsmath}

\usepackage[psamsfonts]{amsfonts}

\usepackage{amsthm}
\usepackage{color}
\usepackage{setspace}

\usepackage[vcentermath,enableskew]{youngtab}
\usepackage{amssymb}
\usepackage{rotating}
\usepackage{young}
\usepackage{enumerate}

\newcommand{\no}{\noindent}
\newcommand\bigzero{\makebox(0,0){\text{\huge0}}}

\newcommand{\be}{\begin{equation}} 
\newcommand{\bea}{\begin{eqnarray}}
\newcommand{\ee}{\end{equation}}
\newcommand{\beas}{\begin{eqnarray*}}
\newcommand{\eea}{\end{eqnarray}}
\newcommand{\eeas}{\end{eqnarray*}}
\newcommand{\non}{\nonumber}

\def\Z{{\mathbb Z}}

\def\F{{\mathbb F}}
\def\Q{{\mathbb Q}}
\def\z3{{\mathbb Z_3}}
\def \zp{{\mathbb Z_p}}
\def \zn{{\mathbb Z_n}}
\def \zq{{\mathbb Z_q}}
\def \zm{{\mathbb Z_m}}

\newtheorem{theorem}{Theorem}[section]
\newtheorem{definition}[theorem]{Definition}

\newtheorem{corollary}[theorem]{Corollary}

\renewcommand{\baselinestretch}{1.3}
\normalsize

\begin{document}
\begin{center}
\large
{\bf On H-Spaces and a Congruence of Catalan Numbers}

\normalsize

\vskip .3cm

Tamar Friedmann$^{1,3}$ and John R. Harper$^2$

\vskip .3cm
\footnotesize

{\it $^1$ Department of Mathematics and Statistics, Smith College

$^2$Department of Mathematics, University of Rochester

$^3$ Department of Physics and Astronomy, University of Rochester

}

\normalsize

\vskip .3cm

{\bf Abstract }
\end{center}

\renewcommand{\baselinestretch}{1.2}
%\setstretch{1.0}
\small
For $p$ an odd prime and $F$ the cyclic group of order $p$, we show that  
 the number of conjugacy classes of embeddings of $F$  in $SU(p)$ such that no element of $F$ has 1  as an eigenvalue is $(1+C_{p-1})/p$,  where $C_{p-1}$ is a Catalan number.  We prove that the only coset space $SU(p)/F$ that admits a $p$-local $H$-structure is the classical Lie group $PSU(p)$. We also show that $SU(4)/\z3$, where $\z3$ is embedded off the center of $SU(4)$,
is a novel example of an $H$-space, even globally. We apply our results to the study of homotopy classes of maps from $BF$ to $BSU(n)$. 

\vskip .2cm
\no {\sl 2010 MSC Codes:} 05A15, 55P45, 05E15, 11A07, 11B50.

\section{Introduction}
\renewcommand{\baselinestretch}{1.3}
\normalsize

In \cite{FS} conjugacy classes of elements of finite order dividing $m$ in $SU(n)$ are identified and counted.  That result poses the question whether any of these elements can be used to produce new finite $H$-spaces with non-trivial fundamental group.

Let $F$ denote a subgroup of $SU(n)$ that is cyclic of finite order dividing $m$. Since $SU(n)$ is the union of conjugacy classes of a maximal torus, we may assume that $F$ is contained in a maximal torus of $SU(n)$.
In order for  a coset space $SU(n)/F$ to admit the structure of an $H$-space, there are restrictions on the values of $n$ and $m$. The basic topological restriction comes from W. Browder's work on differential Hopf-algebras \cite{B}.  In particular, 
if $p$ is a prime that divides $m$, then 
from Theorem 4.7 of \cite{B} it follows that 
the rational cohomology of $SU(n)$ must have a generator in some dimension of the form $2p^f-1$.  Recall that the generators for $SU(n)$ have odd dimensions $2k-1$ with $k$ between $2$ and $n$.  Thus, the minimal case that satisfies the restriction  is $f=1$ and $m=n=p$, and we consider this case, as well as the case $f=1$, $m=3$, $n=4$.

When $m=n=p$, we show that it is only when $F$ is the center of $SU(p)$ that $SU(p)/F$ is an $H$-space (Theorem \ref{onlycenter}), yet we find a new  example of an $H$-space when  $m=3$ and $n=4$ (Theorem \ref{z3su4}). 
As far as we know, 
previous examples of non-classical finite (possibly $p$-local) $H$-spaces with a non-trivial fundamental group have used the center of $SU(n)$ as an ingredient \cite{H, KK}. But the center of $SU(4)$ is $\Z _4$, so our example is off-center.

Borel's structure theorem for $H$-spaces, which plays a central role in the proof of Theorem \ref{onlycenter}, requires that when $m=p$, for $SU(p)/F$ to be an $H$-space the embedding $t$ of $F$ in $SU(p)$  is such that no eigenvalue of $t(g)$, $g\in F$ is equal to 1. Generalizing the eigenvalue condition to any $SU(n)/\zm$ leads us to count, a-la 
\cite{FS}, the number of conjugacy classes of elements of finite order dividing $m$ in $SU(n)$ with no eigenvalue equal to $1$, for any pair of positive integers $m$ and $n$ (Theorem \ref{formulathm}). We call these {\sl special} conjugacy classes.

The combinatorics for the class of cases where $n=m=p$ for $p>3$ yields many examples of embeddings of $F$ not conjugate to the center of $SU(p)$, for which -- as we show -- $SU(p)/F$ is not an $H$-space. The combinatorics also contains a pleasant surprise; the number of special conjugacy classes of $\zp$ subgroups in $SU(p)$ is given by $(1+C_{p-1})/p$ where $C_n$ is the $n$-th Catalan number (Theorem \ref{ppformulathm}; a related observation involving a pair of distinct primes $p$ and $q$ appears in Theorem \ref{pqformulathm}.). 
As far as we know, the first observation that this expression yields an integer appears in OEIS \#A098796, submitted by F. Chapoton; our result may be viewed as a proof of this fact. The related fact that $p\nmid C_{p-1}$ is well-known \cite{AK}. 
 
\enlargethispage{\baselineskip}
The combinatorial count in Theorem \ref{formulathm} is of additional interest
based on a theorem proved independently by W. Dwyer and C. Wilkerson \cite{DW} and by A. Zabrodsky \cite{Z} and D. Notbohm \cite{N}.  
The theorem asserts (among other things) that up to homotopy, essential maps of $BF$ to $BSU(n)$ are in one-to-one correspondence with conjugacy classes of non-constant homomorphisms of $F$ to $SU(n)$, where $F$ is a cyclic group of prime
order $p$.  Hence Theorem \ref{formulathm} and its corollary count the number of essential maps of $BF$ to $BSU(n)$ that do
not factor through $BSU(n-1)$.  We draw attention to the fact that other results in \cite{FS} may be interpreted in a similar manner.

As of this writing, the general question of whether coset spaces of the form $ SU(n)/F$ are $H$-spaces for any $n$ and for $F$ a cyclic group of any order $m$
remains open. We do expect, however, that in addition to the $SU(4)/\z3$ that we already found there are other examples of $H$-spaces of the form $SU(n)/F$ that do not use the center of $SU(n)$ as an ingredient.

Throughout this paper, $\zn$ shall denote the cyclic group of order $n$.

\section{The coset spaces $SU(p)/F$}\label{ProofH}

The Borel structure theorem for finite $H$-spaces \cite{Bo} states that for $p$ an odd prime, the mod $p$ cohomology algebra is a tensor product of exterior algebras on odd degree  elements and a truncated polynomial algebra on even degree elements with truncation at a power of $p$. We use it in what follows.

\begin{theorem} \label{onlycenter}
Let $p$ be an odd prime and let $F\subset SU(p)$ be a cyclic subgroup of order $p$. Then $H^*(SU(p)/F ; \F_p)$ satisfies the Borel structure theorem for finite $H$-spaces if and only if $F$ is the center of $SU(p)$. 
\end{theorem}

\begin{proof} Let $\omega = e^{2\pi i/p}$ and let $g$ be a generator for $F$. 
 We describe $F\subset SU(p)$ in terms of the following diagram. 

%\[
%\begin{tikzcd}
%g \arrow[d] & F \arrow[d] \arrow[r, "t"] & SU(p) \arrow[drr,hook] \\
%\omega &S^1 \arrow[r,"\Delta"]& S^1 \underbrace{\times \cdots \times}_p S^1 \arrow[r, "\theta" ] & S^1 \underbrace{\times \cdots \times}_p S^1 \arrow[r,hook]&U(p)
%\end{tikzcd},
%\]
\begin{center}
\label{diagram1}
  \raisebox{-0.5\height}{\includegraphics[height=1.1in]{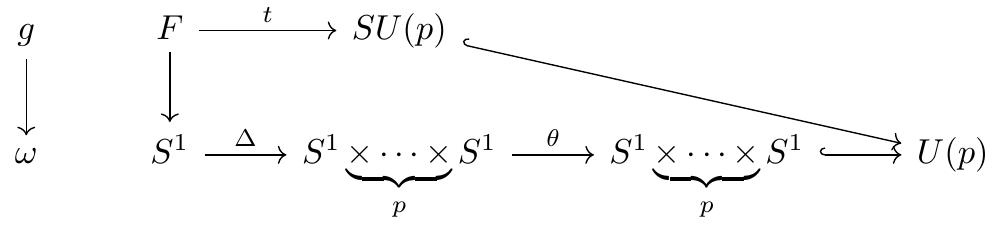}}
\end{center}
\no where $\Delta(z)=(z,z,\ldots , z)$ and $\theta (z_1, \ldots , z_p)=(z_1^{m_1}, \ldots , z_p^{m_p})$ with $\sum _{i=1}^p m_i \equiv 0 \mod p$. 
Then $t (g)$ is the diagonal matrix
\[ \begin{pmatrix} \omega^{m_1} & &  & \bigzero  \\
&\omega^{m_2}& &  \\
&&\ddots &  \\
\bigzero&&&\omega^{m_p}
\end{pmatrix}.
\]

\no We follow the argument on p. 314 of \cite{BB} applied to the commutative diagram of principal fiber bundles 
%\begin{center}
\be \label{cd}
%\label{diagram2}
  \raisebox{-0.5\height}{\includegraphics[height=1.7in]{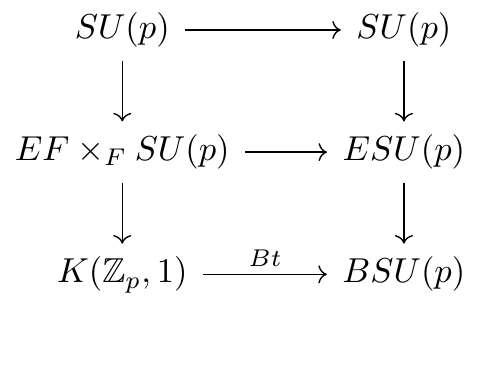}}
%\end{center}
\ee
%\be \label{cd}
%\begin{tikzcd}
%SU(p) \arrow[r] \arrow[d] & SU(p) \dar \\
%EF\times _F SU(p) \dar \rar&ESU(p) \dar \\
%K(\zp , 1) \rar["Bt"]& BSU(p)\\
%\end{tikzcd} ,
%\ee
with the Borel construction in place of $SU(p)/F$ in the left column. Computing $(Bt ) ^*$ on the total Chern class $c$ we have 
\[(Bt)^* (c)=(1+m_1 y_2)(1+m_2 y_2) \cdots (1+m_p y_2)=1+ \sum _{k=1}^p \sigma_k(\overline m)y_2^k\]
\vskip -.3cm
\no and
\[ (Bt)^*c_k = \sigma _k (\overline m)y_2^k\; ,\]
where $\sigma _k$ is the $k$-th elementary symmetric polynomial in $p$ variables, $\overline m = (m_1, m_2, \ldots , m_p)$, $c_k$ is the $k$-th Chern class, and $y_2$ generates $H^2(K(\zp  , 1); \F _p)$. 

Now $(Bt)^* c_1=0$ because $\sigma_1(\overline m)=\sum m_i \equiv 0 \mod p$ as $F\subset SU(p)$.  The first non-trivial differential in the Serre spectral sequence for the left side of diagram (\ref{cd}) is determined by the smallest value of $k$ where $(Bt)^* c_k\neq 0$. If 1 were an eigenvalue of $t(g)$ (i.e some $m_i\equiv 0 \mod p$) then $k\leq p-1$. Assuming $k\leq p-1$, we next show that the Borel structure theorem is not satisfied by the mod $p$ cohomology of $SU(p)/F$. 

We have 
\begin{align*}
E_2^{*,*}&=H^*(K(\zp, 1);\F_p)\otimes H^*(SU(p);\F_p) \non \\ 
&\equiv  \Lambda (x_1)\otimes \zp [y_2] \otimes \Lambda (f_3, \ldots , f_{2p-1}),
\end{align*}
\vskip -.4cm
\no and our hypothesis is 
\[d_{2k-1}(f_{2k-1})=\lambda y_2^k  \hskip .3cm \text{ where } \lambda \not\equiv 0 \mod p.\]
\vskip -.2cm
\no Then 
\[ E_{2k}^{*,*}=\Lambda (x_1)\otimes \zp [y_2]/y_2^k \otimes \Lambda (f_3,\ldots , \hat f_{2k-1},  \ldots , f_{2p-1})\]
omitting the generator $f_{2k-1}$ , 
and 
\[ E_{2k}^{*,*}=E_{\infty}^{*,*}=E_{\infty}^{*,0}\otimes E_{\infty}^{0,*}\]
for dimension reasons and by inspection. The bidegrees of the generators are 
\vskip -.2cm
\[ x_1\in (1,0), \hskip .5cm y_2\in (2,0), \hskip .5cm f_i\in (0, 2i-1) .\]
Thus the mod $p$ cohomology algebra of the total space is generated by elements in these bidegrees. It follows that this algebra is the homology of the free algebra on these generators subject to the differential $d_{2k-1}f_{2k-1}=y_2^k$. If $k\neq p$, this algebra fails to satisfy Borel's theorem. If all the differentials are $0$ then $E_{\infty}^{*,*}$ is not a finite algebra. Thus, $f_{2p-1}$ transgresses to a non-zero multiple of $y_2^p$ and $m_i$ is not congruent to $0 \mod p$ and likewise for the $p$-th symmetric product.

Thus 1 is not an eigenvalue of $t(g)$, $\sigma _k(\overline m)\equiv 0 \mod p$ for $1\leq k \leq p-1$, and $\sigma _p(\overline m)\not\equiv 0 \mod p$. 

It now follows from the unique factorization for polynomials in $\F_p(x)$ that $m_1=m_2 = \cdots = m_p$, so that $F$ is the center of $SU(p)$.
\end{proof}

\begin{corollary} The only quotient of $SU(p)$ by a finite group of order $p$ that yields an $H$-space is $PSU(p)$. 
\end{corollary}

\section{Embeddings of $\zm$ in $SU(n)$ and Catalan numbers} \label{counting}

Here we count the number of ways in which an element of order $m$  can be embedded in $SU(n)$ such that 1 is not an eigenvalue. For the case where $m$ and $n$ are both prime, we also count the number of ways the group $\zm$ can be embedded in $SU(n)$ with
the condition that 1 is not an eigenvalue of any generator of $\zm$.

Let $N'(SU(n),m)$ be the number of conjugacy classes of elements of $SU(n)$ of order $m$, none of whose eigenvalues is 1.\footnote{The notation $N(SU(n),m)$ is reserved for  the number of conjugacy classes of all elements of $SU(n)$ of order $m$ (including those with eigenvalue $1$); it has been computed in \cite{Dj, FS}.}

Let $F_1$, $F_2$ be $\zm$ subgroups of $SU(n)$. We say $F_1$ and $F_2$ are in the same conjugacy class if there exists an element $g\in SU(n)$ such that for any $h\in F_1$, $ghg^{-1}\in F_2$.

\begin{definition} A conjugacy class of $\zm$ subgroups of $SU(n)$ such that no element of the subgroups has eigenvalue 1 is called  {\sl special}. The number of such classes is denoted $SpCG(SU(n), \zm)$.
\end{definition}

Note that each conjugacy class, including special ones, has at least one representative all of whose elements are diagonal. 

\begin{theorem}\label{formulathm} For any positive integers $m$ and $n$, we have 
\be \label{formula} N'(SU(n),m)= {1\over m} \left ( \sum _{d|(m,n)}\phi(d) {m/d+n/d-1\choose n/d}-\sum_{d|(m,n-1)}\phi(d){m/d+{n-1\over d}-1\choose {n-1\over d}}\right ) ,\ee
where $\phi$ is Euler's totient function. 
\end{theorem}

\begin{proof} Let 
$$ F(x,t)=\prod _{k=1}^{m-1} \left (\sum_{a=0}^\infty (t^k x)^a \right ).
$$
A typical term in $F(x,t)$ is 
\[x^{\sum n_k}\, t^{\sum kn_k} ,\] 
where $n_k$, $k=1, \ldots , m-1$ are non-negative integers. If $\sum n_k=n$ and $\sum
kn_k\equiv 0$ mod $m$ then the sequence $\{ n_k\}$ corresponds to a
diagonal $SU(n)$ matrix of order $m$ with eigenvalue $e^{2\pi ik/ m}$ repeated $n_k$ times. Since we excluded $k=0$, the matrix does not have $1$ as an eigenvalue. Thus, such a sequence $\{ n_k\}$ corresponds to a conjugacy class counted by $N'(SU(n),m)$.

To pick out the terms in $F(x,t)$ for which $\sum kn_k\equiv 0$ mod
$m$, thereby obtaining a generating function for $N'(SU(n),m)$, let $\zeta = \exp {2\pi i /m}$ and recall  
\[{1\over m}\sum _{j=0}^{m-1}\zeta ^{jb}=\left \{ \begin{array}{l} 1,\
  \mbox{ if } m \mid b  \\ 0,\ \mbox{ else} \end{array} \right . .
\]
Define
$$ G(x)={1\over m}\sum_{j=0}^{m-1}F(x,\zeta^j).$$
A typical term in $G(x)$ is now
\[ {1\over m}\left (x^{\sum n_k}\, \zeta^{j\sum kn_k}\right ) ,\] 
so the sum over $j$ picks out the terms with $\sum kn_k\equiv 0$ mod $m$, giving
$$ G(x)=\sum_n N'(SU(n),m)x^n\, .$$
We have then
$$G(x)={1\over m}\sum_{j=0}^{m-1}\prod _{k=1}^{m-1} \left (\sum_{a=0}^\infty (\zeta^{jk} x)^a \right ) = {1\over m}\sum_{j=0}^{m-1}\prod _{k=1}^{m-1} \left ({1\over 1-\zeta^{jk}x}\right ).$$
\no The factorization $1-x^d=\prod_{l=0}^{d-1} (1-\zeta^{jl}x)$ for $\zeta ^j$ a primitive $d^{th}$ root of unity gives 
$$ \prod_{\ell=0}^{m-1} (1-\zeta^{j\ell}x)=(1-x^d)^{m/d}= (1-x)\prod_{\ell=1}^{m-1}(1-\zeta^{j\ell}x).
$$
Together with the fact that $\zeta ^j$, $j=0, \ldots, m-1$ is a primitive $d^{th}$ root of unity
$\phi (d)$ times, we obtain
\beas G(x)&=&{1\over m}\sum_{d|m}\phi(d){1-x\over (1-x^d)^{m/d}}\\ &=& {1\over m}\sum_{d|m}\phi(d) (1-x)\sum_{b\geq 0}{m/d+b-1\choose b}x^db\\ &=&{1\over m}\sum_{d|m}\phi(d) \left [ \sum _{b\geq 0}{m/d+b-1\choose b}x^{db}-\sum _{b\geq 0}{m/d+b-1\choose b}x^{db+1} \right ].
\eeas
The coefficient of $x^n$ in $G(x)$ will come from values of $d$ such that $db=n$ in the first term and $db+1=n$ in the second term, leading to $d\mid n$ and $d\mid n-1$, respectively, as in equation (\ref{formula}).
\end{proof}

\begin{corollary}\label{nqnumber} If $p$ is a prime and $n$ is any positive integer, then 
\[  N'(SU(n),p)=  {1\over p} \left [{p+n-2\choose n} + (p-1)\alpha (n,p)\right ] 
\]
where  
\[ \alpha (n,p)=\begin{cases} 1 &\mbox{if } p\mid n \; ,\\
-1 & \mbox{if } p \mid n-1 \; , \\
0& \mbox{otherwise.} \end{cases}
\]

\end{corollary}

\begin{proof} Follows from equation (\ref{formula}).
\end{proof}

In particular, if $p$ and $q$ are distinct primes, we have

\be  N'(SU(p),q)= \left \{ \begin{array}{ll} {1\over q}{q+p-2\choose p} & \text{if } q \nmid  ( p-1) \\  \label{pq}\\
{1\over q}\left [ {q+p-2\choose p}-(q-1)\right ] & \text{if } q\mid (p-1)  \end{array} \right. ,\ee
and if $q=p$, we have
%In particular, if $n=p$, we have 
\be \label{ppnumber} N'(SU(p),p)={1\over p} \left [ p-1+ {2p-2\choose p} \right ]. \ee

\begin{theorem}\label{pqformulathm} If $p$ and $q$ are distinct primes
%revised:
and $p\nmid q-1$ then

\be \label{pandqeqn} SpCG(SU(p), \zq)= \left \{ \begin{array}{ll} {1\over (q-1)q}{q+p-2\choose p}={(q+p-2)!\over p!q!} & \text{if } q \nmid  ( p-1) , \\ \\
{1\over (q-1)q}\left [ {q+p-2\choose p}-(q-1)\right ] & \text{if } q\mid (p-1) .
 \end{array} \right. \ee
\vskip .2cm
%revised except in the case $p=2$, $q=3$, where $CG(SU(2), \z3)=1$
\no Further, if $p=2$ and $m$ is any positive integer, $SpCG(SU(2), \Z_m)=1$. 
\end{theorem}

\begin{proof} We first show that 
%in almost all cases, 
all the $(q-1)$ generators of all $\zq$ groups are in distinct conjugacy classes of $SU(p)$.  %In those cases, 
It follows that we may divide the expressions in equation (\ref{pq}) by $(q-1)$ and obtain the result. 
%As we shall see, the only case where this does not hold is $p=2$, $q=3$, which we treat separately. 

Let $F$ be a $\zq$ subgroup of $SU(p)$, and let $h\in F$. Let $\zeta = \exp{2\pi i /q}$ and let $n_k$, $k=1, \ldots , q-1$, be the multiplicity of $\zeta ^k$ as an eigenvalue of $h$ (so $\sum n_k = p, \sum kn_k\equiv 0 \mod q$). 

Suppose $h^t$ is in the same conjugacy class as $h$ for some integer $t$. Then the multiplicity of the eigenvalues of $h^t$ is the same as that for $h$. It follows that $n_{kt}=n_k$ for all $k$ (all indices are taken mod $q$). It also follows that $h^{t^l}$ is in the same conjugacy class as $h$ for any integer $l$ (if $h^t=ghg^{-1}$ for some $g\in SU(p)$, then  $h^{t^l} = g^lhg^{-l}$). Therefore, $n_{kt^l}=n_k$ for any $l$ and $k$. 

Let $c$ be the order of $t$, i.e. the smallest integer such that $t^c\equiv 1 \mod q$ ($c$ exists because $q$ is prime). We have  $n_1=n_t=n_{t^2}=\cdots =n_{t^{c-1}}$, all indices being distinct $\mod q$. Let $A_{i_1}=\{1, t, t^2, \ldots , t^{c-1}\} $.  Now pick any $i_2\notin A_{i_1} $ and let $A_{i_2}=\{i_2, i_2t, i_2t^2, \ldots , i_2t^{c-1} \}$. Then $A_{i_2}$ has $c$ distinct elements $\mod q$, as $i_2t^{\ell_1}\equiv i_2t^{\ell_2} \mod q$ implies  $\ell_1-\ell_2\geq c$. Further, $A_{i_2}$ does not intersect with $A_{i_1}$, as $t^{\ell_1}\equiv i_2t^{\ell_2} \mod q$ implies $i_2\equiv t^{\ell_1-\ell_2} \mod q$, i.e. $i_2\in A_{i_1}$, a contradiction. We see that the set of $n_k$'s can be partitioned into several subsets of size $c$, with all the $n_k$'s  within each subset equal to each other. 

\enlargethispage{\baselineskip}
Since $\sum _k n_k = p$, it follows that $p$ is divisible by $c$, so $c=1$ or $c=p$. If $c=1$ then $t\equiv 1 \mod q$ and there are no conjugacies. If $c=p$ then 
%revise rest of proof
since $c$ is the order of a subgroup of $\Z_q^*$, we have $c\mid q-1$, so $p\mid q-1$ which we assumed is not the case. 

For $p=2$, if $m$ is any positive integer there is exactly one conjugacy class of a $Z_m$ subgroup generated by $\text{diag}(\zeta, \zeta^{m-1})$ where $\zeta=e^{2\pi i\over q}$. So $SpCG(SU(2),\Z_m)=1$. 
%$n_1=n_2=\cdots = n_{q-1}$ so $p=(q-1)n_1$. Hence either $n_1=p$, in which case  $q=2$, or $n_1=1$, in which case $p=q-1$ leading to $p=2$, $q=3$. 
%
%For the case $c=p$ and $q=2$, the only possible generator with no eigenvalue 1 is $h=\text{diag}(-1, -1, \ldots , -1)$. But this is an element of $SU(p)$ only if $p$ is even, i.e. $p=2$. So $c=p=2$, leading to $t^2\equiv 1 \mod 2$ which requires $t=1$ so there are no conjugacies. (Note: this case has $p=q$ and is actually covered in the next theorem). 
%
%For the case $p=2$, $q=3$, there is one conjugacy class of $\z3$ subgroups of $SU(2)$, generated by $h=\text{diag} (\zeta, \zeta^2)$ where $\zeta=e^{2\pi i/3}$ (note, $h^2$ is conjugate to $h$).
\end{proof}
\begin{theorem} \label{ppformulathm}If $p$ is prime, then

\be \label{pp} SpCG(SU(p), \zp)= {1\over p} (1+C_{p-1}) \, , \ee
where 
\[ C_{p-1}={1\over p}{2p-2 \choose p-1}
\]
is the $(p-1)$th Catalan number. 

\end{theorem}

\begin{proof} The same proof as the previous theorem holds, with $q$ replaced by $p$.
\end{proof}

\vskip .5cm
\begin{corollary} \label{catalanequiv} For any prime $p$, we have
\be  C_{p-1} \equiv -1 \mod p \, .\ee
\end{corollary}
\begin{proof} Follows from Eq. (\ref{pp}). 
\end{proof}

%\newpage
\no For all $p\geq 5$, the number given by Eq. (\ref{pp}) is larger than 1 and blows up quickly: for $p=5,7,11, \ldots $ it is $3, 19, 1527, \ldots $ .\footnote{As it happens, this series appears in the Online Encyclopaedia of Integer Sequences as A098796, submitted there by F. Chapoton in 2004 \cite{Ch}.} For $p=5$, the resulting pair of subgroups are generated by diag$(\zeta, \zeta, \zeta, \zeta^3, \zeta^4)$ and diag$(\zeta, \zeta, \zeta^2, \zeta^2, \zeta^4)$, where $\zeta = e^{2\pi i/ 5}$. Generators for the 18 groups at $p=7$  appear in the appendix.

\begin{corollary} \label{pandqequiv}For any primes $q$ and $p$ such that $q\; |\; p-1$, 
\be {1\over q-1}{q+p-2\choose p} \equiv 1 \mod q . \ee
\[ \] 
\end{corollary}

\begin{proof} Follows from Eq. (\ref{pandqeqn}) for $q\, |\, p-1$.
\end{proof}

\vskip .3cm
\no Both corollaries can be checked using Wilson's theorem, $(p-1)!\equiv -1 \mod p$. 

\section{Homotopy classes of maps}
Our results may be applied to the study of homotopy classes of maps $B\zp \rightarrow BSU(n)$. Theorems proved independently by  W. Dwyer and C. Wilkerson \cite{DW} and by A. Zabrodsky \cite{Z} and D. Notbohm \cite{N} say that homotopy classes of essential maps are in one-to-one correspondence with conjugacy classes of non-constant homomorphisms of $Z_p$ to $SU(n)$. These in turn are determined by conjugacy classes of non-identity elements of prime order $p$. To connect our results with this theory, we make a definition. 

\begin{definition} A homotopy class of essential maps from  $B\zm$ to $BSU(n)$ is called {\normalfont special} if it does not factor through $BSU(n-1)$. 
\end{definition}
Corresponding to the special conjugacy classes of elements counted by Theorem \ref{formulathm}  and Corollary \ref{nqnumber}, we have the  special homotopy classes of maps for $m=p$. Then the following is a restatement of Corollary \ref{nqnumber}.

\begin{corollary}\label{essentialnq} 
For any prime $p$ and any integer $n$, the number of special homotopy classes of essential maps from $B\zp$ to $BSU(n)$ is 
\[   {1\over p} \left [{p+n-2\choose n} + (p-1)\alpha (n,p) \right ]  ,
\]
where  
\[ \alpha (n,p)=\begin{cases} 1 &\mbox{if } p\mid n \; ,\\
-1 & \mbox{if } p \mid n-1 \; , \\
0& \mbox{otherwise.} \end{cases}
\]
In particular, the number of special homotopy classes of essential maps from $B\zp$ to $BSU(p)$ is
 \[ {1\over p} \left [ p-1+ {2p-2\choose p} \right ]. \]
\end{corollary}

\section{The space $ SU(4)/\z3$}

In this section, let $F=\z3$ be the center of $SU(3)$. Embed $F$ in $SU(4)$ via the standard inclusion of $SU(3)$ in $SU(4)$ given in terms of matrices by 
\[ A \hookrightarrow \begin{pmatrix} A&0 \\ 0&1 \end{pmatrix} . \]
We have a diagram of inclusions of subgroups yielding $S^5$ and $S^7$ as homogeneous spaces in two ways,

%\newpage
%\[
%\begin{tikzcd}
%S^3 \arrow[r, hook] \arrow[d,hook]  
%&SU(3) \arrow[d, hook] \arrow[r] & S^5\\ 
%Sp(2) \arrow[d] \arrow[r, hook]&SU(4) \arrow[r] \arrow[d] &S^5\\ 
%S^7 & S^7& 
%\end{tikzcd}
%\]
\begin{center}
 \raisebox{-0.5\height}{\includegraphics[height=1.5in]{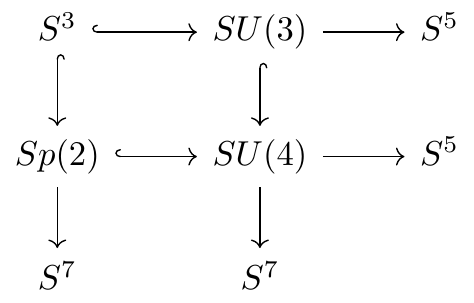}}
\end{center}
%\newpage
\no We regard the spheres as right coset spaces. Taking left coset spaces with respect to $F$, we obtain the diagram
%\be \label{z3su4diag}
%\begin{tikzcd} 
%S^3 \rar \dar[hook] &PSU(3)\rar \dar & L(5;1,1,1) \\
%Sp(2)\rar[hook] \dar &SU(4)/F \rar \dar & L(5;1,1,1)\\
%S^7& S^7
%\end{tikzcd}
%\ee

\begin{equation}\label{z3su4diag}
 \raisebox{-0.5\height}{\includegraphics[height=1.4in]{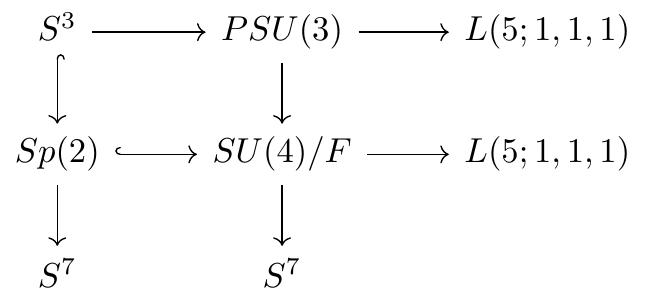}}
\end{equation}
where the lens space $L(5;1,1,1)$ is the quotient of the 5-sphere by the cyclic group of order 3.

In \cite{H} p. 332, it is proven that after localizing at $p=3$, the top row of diagram (\ref{z3su4diag}) splits;  hence $L(5;1,1,1)$ is a 3-local $H$-space and $SU(4)/F$ is 3-locally equivalent to $Sp(2)\times L(5;1,1,1)$ which is a 3-local $H$-space. Rationally, $SU(4)/F$ is $K(\Q,3)\times K(\Q,5)\times K(\Q,7)$ which can only be primitively generated. Since the action by $F$ is via maps homotopic to the identity, the spaces are simple and we may apply localization there to conclude that $SU(4)/F$ is an $H$-space since each of its $p$-localizations are $H$-spaces and each rationalization map is an $H$-map. So we have proved:
\begin{theorem}\label{z3su4} The coset $SU(4)/\z3 $, with $\z3$ embedded via inclusion of the center of $SU(3)$ in $SU(4)$, is an $H$-space.
\end{theorem}

\vskip .5cm
\no {\bf Acknowledgements}

We thank Richard Stanley and Fr\'ed\'eric Chapoton for helpful discussions. We also thank the referee for helpful comments. 

%\vskip 1cm

\begin{appendix}
\section{Generators for subgroups of $SU(7)$ of order $7$.}
Below we list generators for the 18 subgroups of $SU(7)$ of order 7 (not including the center). 
A set of 7 integers $[a_1,  \ldots , a_7]$ corresponds to the generator diag$(\zeta ^{a_1}, \ldots , \zeta^{a_7})$ where $\zeta = e^{2\pi i/7}$. 

\vskip .5cm

\no [1, 2, 2, 2, 2, 2, 3], [1, 1, 2, 2, 2, 3, 3], [1, 1, 1, 2, 3, 3, 3], [1, 1, 2, 2, 2, 2, 4], 

\no [1, 1, 1, 2, 2, 3, 4], [1, 1, 1, 1, 3, 3, 4], [2, 3, 3, 3, 3, 3, 4], [2, 2, 3, 3, 3, 4, 4], 

\no  [1, 3, 3, 3, 3, 4, 4], [2, 2, 2, 3, 4, 4, 4], [1, 2, 3, 3, 4, 4, 4], [1, 1, 3, 4, 4, 4, 4],

\no  [1, 1, 1, 1, 2, 3, 5], [2, 2, 2, 3, 3, 4, 5], [1, 2, 3, 3, 3, 4, 5],  [1, 2, 2, 3, 4, 4, 5],

\no [1, 1, 3, 3, 4, 4, 5], [1, 1, 2, 3, 4, 5, 5]. 

\end{appendix}
\renewcommand{\baselinestretch}{1.0}

\end{document}